\documentclass[12pt,a4paper]{article}
\usepackage{amsfonts}
\usepackage{amsmath}
\usepackage[onehalfspacing,nodisplayskipstretch]{setspace}
\usepackage{eucal}

\newtheorem{theorem}{Theorem}

\newenvironment{proof}[1][Proof]{\noindent\textbf{#1.} }{\ \rule{0.5em}{0.5em}}
\input{tcilatex}
\begin{document}

\title{Large Deviations Inequalities for Unequal Probability Sampling
Without Replacement\thanks{%
This note answers a question posed by Noam Nisan. We thank Noam Nisan and
Benji Weiss for useful discussions and suggestions.}}
\author{Dean P. Foster\thanks{%
Department of Statistics, Wharton, University of Pennsylvania, Philadelphia,
and Amazon, New York. \emph{e-mail}: \texttt{dean@foster.net} \ \emph{web
page}: \texttt{http://deanfoster.net}} \and Sergiu Hart\thanks{%
Institute of Mathematics, Department of Economics, and Federmann Center for
the Study of Rationality, The Hebrew University of Jerusalem. \emph{e-mail}: 
\texttt{hart@huji.ac.il} \ \emph{web page}: \texttt{%
http://www.ma.huji.ac.il/hart}}}
\maketitle

\begin{abstract}
We provide bounds on the tail probabilities for simple procedures that
generate random samples \emph{without replacement}, when the probabilities
of being selected need not be equal.
\end{abstract}

\def\@biblabel#1{#1\hfill}
\def\thebibliography#1{\section*{References}
\addcontentsline{toc}{section}{References}
\list
{}{
 \labelwidth 0pt
 \leftmargin 1.8em
 \itemindent -1.8em
 \usecounter{enumi}}
\def\newblock{\hskip .11em plus .33em minus .07em}
\sloppy\clubpenalty4000\widowpenalty4000
\sfcode`\.=1000\relax\def\baselinestretch{1}\large \normalsize}
\let\endthebibliography=\endlist%

Let $n\geq k$ be two positive integers. Consider a population of size $n$
with associated \emph{relative} \emph{weights} $0\leq w^{1},...,w^{n}\leq 1/k
$ such that\footnote{%
Superscripts are used as indices for elements (or subsets) of the population.%
} $\sum_{i=1}^{n}w^{i}=1$. We want to choose a random sample \emph{without
replacement} from the population $[n]=\{1,...,n\}$ such the probabilities of
being in the sample are proportional to the weights; i.e., letting $S$
denote the random sample, we require that%
\begin{equation}
\mathbb{P}\left[ i\in S\right] =kw^{i}  \label{eq:prob}
\end{equation}%
for every $i\in \lbrack n]$ (this explains the constraint $w^{i}\leq 1/k$;
the factor $k$ comes from\footnote{%
We write $|F|$ for the number of elements of a finite set $F$.} $\sum_{i\in
\lbrack n]}\mathbb{P}\left[ i\in S\right] =\left\vert S\right\vert =k$).
Moreover, we want to do this in such a way that we obtain \textquotedblleft
concentration" or \textquotedblleft large deviations" inequalities, as in
the case of sampling with replacement. That is, for every subset of the
population $A\subset \lbrack n],$ with high probability its proportion of
the sample $\left\vert S\cap A\right\vert /k$ does not exceed its relative
weight $\alpha =\sum_{i\in A}w^{i}$ by much: for $\delta >0,$ 
\begin{equation*}
\mathbb{P}\left[ \frac{1}{k}\left\vert S\cap A\right\vert \geq \alpha
+\delta \right] 
\end{equation*}%
is exponentially small in $k$ and $\delta $.

For comparison, assume that the sampling is done \emph{with replacement},
i.e., by $k$ i.i.d. draws from the population $[n]$, with probabilites $%
w^{1},...,w^{n}$. Let $N^{i}$ be the number of times that $i$ has been
selected, then $N^{A}:=\sum_{i\in A}N^{i},$ the number of times that an
element of $A$ has been selected, is a \textrm{Binomial}$(k,\alpha )$ random
variable. Therefore, by the Chernoff--Hoeffding inequality (see (\ref%
{eq:ch-ho}) in the Appendix), for every $\delta >0$ we have%
\begin{eqnarray}
\mathbb{P}\left[ \frac{1}{k}N^{A}\geq \alpha +\delta \right] &=&\mathbb{P}%
\left[ N^{A}-k\alpha \geq k\delta \right]  \notag \\
&\leq &\exp \left( -D\left( \left. \alpha +\delta \right\Vert \alpha \right)
\cdot k\right) \leq \exp \left( -2\delta ^{2}k\right) ,  \label{eq:with-repl}
\end{eqnarray}%
where%
\begin{equation*}
D(q||p):=q\ln \frac{q}{p}+(1-q)\ln \frac{1-q}{1-p}
\end{equation*}%
denotes the Kullback--Leibler divergence of from $p$ to $q$ (for $0<p,q<1$).
(When the sampling is done without replacement each $N^{i}$ takes only the
values $0$ and $1,$ and $N^{A}=\left\vert S\cap A\right\vert $.)

There are various methods to generate random samples without replacement so
that condition (\ref{eq:prob}) holds. However, they do not immediately yield
tail probability bounds such as (\ref{eq:with-repl}). We will therefore
consider martingale-based procedures, as they are naturally amenable to
large deviation analysis. Such procedures have been proposed by Deville and
Till\'{e} [1998] (they call them \textquotedblleft splitting methods"); in
particular, Procedure $\mathfrak{X}$ below is their \textquotedblleft
pivotal method."

\section{Martingale-Based Procedures}

It is convenient to rescale the weights so that they add to $k;$ thus, put $%
\Delta :=\{x\in \lbrack 0,1]^{n}:\sum_{i=1}^{n}x^{i}=k\},\ $and let $\Delta
_{0}:=\{x\in \{0,1\}^{n}:\sum_{i=1}^{n}x^{i}=k\}$ be the set of extreme
points of $\Delta ,$ i.e., those weight vectors that contain $k$ ones and $%
n-k$ zeros. For every set $A\subset \lbrack n]$ we write $x^{A}:=\sum_{i\in
A}x^{i}.$ A procedure that generates a random sample $S$ of size $k$ such
that $\mathbb{P}\left[ i\in S\right] =x^{i}$ for every $i\in \lbrack n]$ is
called an $x$\emph{-procedure}. The vector of normalized weights $%
(w^{1},...,w^{n})$ yields the vector of weights $%
x_{0}:=kw=(kw^{1},...,kw^{n})$ in $\Delta .$

We start with a trivial observation.

\textbf{Observation.} \emph{Let }$x=\sum_{\ell =1}^{L}\lambda _{\ell
}x_{\ell }$\emph{\ where }$x_{\ell }\in \Delta $\emph{\ and }$\lambda _{\ell
}\geq 0$\emph{\ for each }$\ell $\emph{, and }$\sum_{\ell =1}^{L}\lambda
_{\ell }=1$\emph{\ (and thus }$x\in \Delta $\emph{\ as well). If }$\mathfrak{%
X}_{\ell }$\emph{\ is an }$x_{\ell }$\emph{-procedure for each }$\ell ,$%
\emph{\ then the procedure }$\mathfrak{X}$\emph{\ that with probability }$%
\lambda _{\ell }$\emph{\ follows }$\mathfrak{X}_{\ell }$\emph{\ is an }$x$%
\emph{-procedure.}

This is immediate by%
\begin{equation*}
\mathbb{P}_{\mathfrak{X}}\left[ i\in S\right] =\sum_{\ell =1}^{L}\lambda
_{\ell }\mathbb{P}_{\mathfrak{X}_{\ell }}\left[ i\in S\right] .
\end{equation*}

Iterating this observation yields a martingale: a stochastic process where
at each step the (conditional) expectation of the \textquotedblleft value"
of the next state equals the \textquotedblleft value" of the current state;
in our case, these \textquotedblleft values" will be weight vectors in $%
\Delta .$

Let thus $(X_{t})_{t=0,1,2,...}$ be a $\Delta $-valued martingale starting
with the constant $X_{0}=x_{0}=kw$ and ending at a finite time $T$ with $%
X_{T}$ in $\Delta _{0}$ (i.e., $X_{T}^{i}$ is either $0$ or $1$ for every $i$%
).\footnote{%
The time $T$ may well be random; for simplicity, once $\Delta _{0}$ is
reached the martingale stays constant, i.e., $X_{t}=X_{T}$ for all $t\geq T.$%
} Since for every $x$ in $\Delta _{0}$ there is a unique $x$%
-procedure---namely, the deterministic choice of the sample as those $i$
whose $x^{i}$ is $1$ (i.e., $S=\{i:x^{i}=1\}$)---the observation above
yields an $x_{0}$-procedure that uses these deterministic $X_{T}$-choices.
Thus, $i$ belongs to the random sample $S$ if and only if $X_{T}^{i}=1$;
and, for every set $A\subset \lbrack n]$, the number of elements of $A$ in
the sample is%
\begin{equation*}
|S\cap A|=\sum_{i\in A}X_{T}^{i}=X_{T}^{A}.
\end{equation*}

The martingale constructions below are based on moving weights around as
much as possible, subject to the constraint that all weights stay between $0$
and $1.$

\section{General Procedures $\mathfrak{X}$}

We describe a class of simple procedures that follow the \textquotedblleft
pivotal method" of Deville and Till\'{e} [1998].

\bigskip

\begin{quotation}
\noindent \textbf{Procedure }$\mathfrak{X}$. Start with $X_{0}=x_{0}.$ At
every step take two indices $i\neq j$ such that\footnote{%
If this is step $t,$ then $x^{i}$ stands for $X_{t-1}^{i}$ and $\tilde{x}%
^{i} $ for $X_{t}^{i}.$} $0<x^{i},x^{j}<1$ (for now the order in which these
steps are carried out is arbitrary, or, alternatively, random). We
distinguish two cases:\footnote{%
The reason that we have conveniently included $x^{i}+x^{j}=1$ in Case 2 will
become clear in the next section.}

\textbf{(i)} If $x^{i}+x^{j}<1$ then either the weight of $j$ is transferred
to $i$ or the weight of $i$ is transferred to $j$: with probability $%
x^{i}/(x^{i}+x^{j})$ the new weights are $\tilde{x}^{i}=x^{i}+x^{j}$ and $%
\tilde{x}^{j}=0,$ and with probability $x^{j}/(x^{i}+x^{j})$ the new weights
are $\tilde{x}^{i}=0$ and $\tilde{x}^{j}=x^{i}+x^{j}$ (the probabilities are
determined by the martingale condition, i.e., the expectation of $\tilde{x}%
^{i}$ being equal to $x^{i}$).

\textbf{(ii)} If $x^{i}+x^{j}\geq 1$ then either $1-x^{i}$ is transferred
from $j$ to $i$ or $1-x^{j}$ is transferred from $i$ to $j$: with
probability $(1-x^{j})/(2-x^{i}-x^{j})$ the new weights are $\tilde{x}^{i}=1$
and $\tilde{x}^{j}=x^{i}+x^{j}-1$, and with probability $%
(1-x^{i})/(2-x^{i}-x^{j})$ the new weights are $\tilde{x}^{i}=x^{i}+x^{j}-1$
and $\tilde{x}^{j}=1$ (the probabilities are again determined by the
martingale condition, i.e., the expectation of $\tilde{x}^{i}$ being equal
to $x^{i}$).
\end{quotation}

\bigskip

We have in fact defined a class of procedures $\mathfrak{X}$, depending on
the order in which the active $i,j$ are chosen at every step; we will see in
the next section that doing so in a consistent way allows proving better
probability bounds.

The procedures $\mathfrak{X}$ have the following properties. First, at each
step (at least) one weight becomes $0$ or $1$ and will no longer change,
which implies that the number of steps $T$ until $\Delta _{0}$ is reached is
at most $n.$ Second, at each step the amount of weight that is moved is at
most $1$.

Put%
\begin{equation*}
\pi (\eta ,\delta ,k):=\left[ \left( \frac{\eta }{\eta +\delta }\right)
^{\eta +\delta }\exp (\delta )\right] ^{k}\leq \exp \left( -\frac{\delta
^{2}/2}{\eta +\delta /3}k\right)
\end{equation*}%
(as we will see below, these are probability bounds given by the Freedman
inequality (\ref{eq:freedman}) in the Appendix).

\begin{theorem}
\label{th:freedman}The procedures $\mathfrak{X}$ satisfy $\mathbb{P}\left[
i\in S\right] =kw^{i}$ for every $i$ in $[n]$, and%
\begin{eqnarray}
\mathbb{P}\left[ \frac{1}{k}\left\vert S\cap A\right\vert \geq \alpha
+\delta \right] &\leq &\pi (\eta ,\delta ,k)\text{\ \ and}
\label{eq:ineq-freedman} \\
\mathbb{P}\left[ \frac{1}{k}\left\vert S\cap A\right\vert \leq \alpha
-\delta \right] &\leq &\pi (\eta ,\delta ,k)  \label{eq:ineq-freedman2}
\end{eqnarray}%
for every set $A\subset \lbrack n]$ and $\delta \geq 0$, where 
\begin{equation*}
\alpha :=\sum_{i\in A}w^{i}
\end{equation*}
is the relative weight of $A$, and 
\begin{equation}
\eta :=\alpha -k\sum_{i\in A}(w^{i})^{2}\leq \alpha .  \label{eq:eta}
\end{equation}
\end{theorem}

\begin{proof}
We will use the Freedman inequality (\ref{eq:freedman}). To do so we need to
bound $V_{T}:=\sum_{t=1}^{T}$\textrm{Var}$(Y_{t}^{A}|\mathcal{F}_{t-1})$,
where $Y_{t}^{i}:=X_{t}^{i}-X_{t-1}^{i}$ and $Y_{t}^{A}:=\sum_{i\in
A}Y_{t}^{i}$ are the corresponding martingale differences. For each $i$ we
have%
\begin{equation}
\sum_{t=1}^{T}\text{\textrm{Var}}(Y_{t}^{i}|\mathcal{F}_{t-1})=\text{\textrm{%
Var}}(X_{T}^{i})=x_{0}^{i}(1-x_{0}^{i}),  \label{eq:var-i}
\end{equation}%
the first equality because $X_{t}^{i}$ is a martingale, and the second
because $X_{T}^{i}\in \{0,1\},$ and so $\mathbb{P}\left[ X_{T}^{i}=1\right]
=X_{0}^{i}=x_{0}^{i}$ (again, since $X_{t}^{i}$ is a martingale).

Next, we claim that for each $t$%
\begin{equation}
\text{\textrm{Var}}(Y_{t}^{A}|\mathcal{F}_{t-1})\leq \sum_{i\in A}\text{%
\textrm{Var}}(Y_{t}^{i}|\mathcal{F}_{t-1}).  \label{eq:var-r}
\end{equation}%
Indeed, let $i,j$ be the active indices at step $t$; if $i$ and $j$ are both
outside $A,$ then the two sides of (\ref{eq:var-r}) equal $0$; if one of
them, say $i$, is in $A$ and the other, $j$, is outside $A$, then the two
sides equal \textrm{Var}$(Y_{t}^{i}|\mathcal{F}_{t-1})$; and if $i$ and $j$
are both in $A$ then the lefthand side is $0$ (because the total change in
the weight of $A$, namely $Y_{t}^{A},$ is zero).

Adding (\ref{eq:var-r}) over $t$ and then using (\ref{eq:var-i}) yields%
\begin{eqnarray}
V_{T} &=&\sum_{t=1}^{T}\text{\textrm{Var}}(Y_{t}^{A}|\mathcal{F}_{t-1})\leq
\sum_{t=1}^{T}\sum_{i\in A}\text{\textrm{Var}}(Y_{t}^{i}|\mathcal{F}_{t-1}) 
\notag \\
&=&\sum_{i\in A}\sum_{t=1}^{T}\text{\textrm{Var}}(Y_{t}^{i}|\mathcal{F}%
_{t-1})=\sum_{i\in A}x_{0}^{i}(1-x_{0}^{i})  \notag \\
&=&\sum_{i\in A}x_{0}^{i}-\sum_{i\in
A}(x_{0}^{i})^{2}=kw^{A}-k^{2}\sum_{i\in A}(w^{i})^{2}=k\eta .
\label{eq:tau}
\end{eqnarray}

The Freedman's inequality (\ref{eq:freedman}), with $c=k\delta $ and $%
v=k\eta $, then gives%
\begin{eqnarray*}
\mathbb{P}\left[ \frac{1}{k}\left\vert S\cap A\right\vert \geq \alpha
+\delta \right] &=&\mathbb{P}\left[ X_{T}^{A}-X_{0}^{A}\geq k\delta \right]
\leq \left( \frac{k\eta }{k\eta +k\delta }\right) ^{k\eta +k\delta }\exp
(k\delta ) \\
&=&\left[ \left( \frac{\eta }{\eta +\delta }\right) ^{\eta +\delta }\exp
(\delta )\right] ^{k}=\pi (\eta ,\delta ,k),
\end{eqnarray*}%
which is (\ref{eq:ineq-freedman}).

Finally, consider the martingale $-X_{t}^{A}$: its differences are $%
-Y_{t}^{A},$ and so the sum of variances is the same $V_{T}$ above; this
yields%
\begin{eqnarray*}
\mathbb{P}\left[ \frac{1}{k}\left\vert S\cap A\right\vert \leq \alpha
-\delta \right] &=&\mathbb{P}\left[ X_{T}^{A}-X_{0}^{A}\leq -k\delta \right]
\\
&=&\mathbb{P}\left[ \left( -X_{T}^{A}\right) -\left( -X_{0}^{A}\right) \geq
k\delta \right] \leq \pi (\eta ,\delta ,k),
\end{eqnarray*}%
which is (\ref{eq:ineq-freedman2}).
\end{proof}

\section{\textquotedblleft In Order" Procedures $\mathfrak{X}^{\ast }$}

One issue with the general procedures $\mathfrak{X}$ is that their number of
steps is $n$ rather than $k$ (as in sampling with replacement); the Freedman
inequality circumvents this by counting conditional variances rather than
steps. We will now see that running a procedure $\mathfrak{X}$
\textquotedblleft \emph{in order}" (see below; we call this \emph{Procedure }%
$\mathfrak{X}^{\ast }$) allows combining sequences of consecutive steps into
\textquotedblleft \emph{rounds}," in such a way that the number of rounds is
at most $k$, and the martingale that corresponds to these rounds (while
ignoring the individual steps) enjoys the same properties as the original
(step-based) martingale used in the previous section.

\bigskip

\begin{quotation}
\noindent \textbf{Procedure }$\mathfrak{X}^{\ast }$\emph{.} Fix an order on $%
[n],$ say the natural order. Run procedure $\mathfrak{X}$ with the active $i$
and $j$ that exchange weights at each step $t$ being the \emph{minimal}
\textquotedblleft yet undecided" indices, i.e., the minimal $i\neq j$ with%
\footnote{%
See Remarks (a) and (b) below.} $0<X_{t-1}^{i},X_{t-1}^{j}<1.$
\end{quotation}

\bigskip

Put%
\begin{equation}
\pi ^{\ast }(\eta ,\delta ,k):=\exp \left( -D\left( \frac{\eta +\delta }{%
1+\eta }\left\Vert \frac{\eta }{1+\eta }\right. \right) \cdot k\right) .
\label{eq:pi*}
\end{equation}%
Comparing with sampling with replacement, where we had $D(\alpha +\delta
||\alpha )$ in (\ref{eq:with-repl}), we now have $D(\tilde{\alpha}+\tilde{%
\delta}||\tilde{\alpha})$ in (\ref{eq:pi*}) with $\tilde{\alpha}=\alpha
/(1+\alpha )$ and $\tilde{\delta}=\delta /(1+\alpha )$ (we took $\eta
=\alpha $ as a worst case; see the next section).

\begin{theorem}
\label{th:fgl}The procedures $\mathfrak{X}^{\ast }$ satisfy $\mathbb{P}\left[
i\in S\right] =kw^{i}$ for every $i$ in $[n]$, and%
\begin{eqnarray*}
\mathbb{P}\left[ \frac{1}{k}\left\vert S\cap A\right\vert \geq \alpha
+\delta \right] &\leq &\pi ^{\ast }(\eta ,\delta ,k)\text{\ \ and} \\
\mathbb{P}\left[ \frac{1}{k}\left\vert S\cap A\right\vert \leq \alpha
-\delta \right] &\leq &\pi ^{\ast }(\eta ,\delta ,k)
\end{eqnarray*}%
for every set $A\subset \lbrack n]$ and $\delta \geq 0$, where $\alpha
:=\sum_{i\in A}w^{i}$ is the relative weight of $A$, and $\eta :=\alpha
-k\sum_{i\in A}(w^{i})^{2}\leq \alpha .$
\end{theorem}

\begin{proof}
When running a procedure $\mathfrak{X}^{\ast },$ consider the times where
some weight $i$ becomes $1$ (which implies that $i$ will for sure be in the
sample $S$). This happens $k$ times,\footnote{%
If some starting weight $x_{0}^{i}$ equals $1$ then $i$ must be in the
sample $S,$ and so in general there are $k_{1}:=k-\left\vert
\{i:x_{0}^{i}=1\}\right\vert \leq k$ rounds. For simplicity assume that $%
x_{0}^{i}<1$ for all $i$ in $N$, and thus $k_{1}=k$.} specifically, whenever
we are in case (ii) (i.e., when $x^{i}+x^{j}\geq 1;$ note that only one
weight can become $1$ at a time, because $x^{i},x^{j}<1$ implies $%
x^{i}+x^{j}<2$). Let $T_{0}\equiv 0<T_{1}<T_{2}<...<T_{k}\equiv T$ be these
(random) stopping times, and put $Z_{\ell }:=X_{T_{\ell }}$ for $\ell
=0,1,...,k.$ Then $(Z_{\ell })_{\ell =0,...,k}$ is a martingale, $%
Z_{0}=x_{0},$ and $Z_{k}=X_{T}\in \Delta _{0}.$ Moreover, the sum of the
conditional variances is the same for the two martingales: for every set $%
A\subset \lbrack n]$ we have%
\begin{eqnarray*}
\sum_{\ell =1}^{k}\mathrm{Var}\left[ Z_{\ell }^{A}-Z_{\ell -1}^{A}|\mathcal{F%
}_{T_{\ell -1}}\right] &=&\sum_{\ell =1}^{k}\sum_{t=T_{\ell -1}+1}^{T_{\ell
}}\mathrm{Var}\left[ Y_{t}^{A}|\mathcal{F}_{t-1}\right] \\
&=&\sum_{t=1}^{T}\mathrm{Var}\left[ Y_{t}^{A}|\mathcal{F}_{t-1}\right] \leq
k\eta ,
\end{eqnarray*}%
where the inequality is by (\ref{eq:tau}).

Next, we claim that the total weight that is moved in each round is at most $%
1;$ i.e., 
\begin{equation}
\left\vert Z_{\ell }^{A}-Z_{\ell -1}^{A}\right\vert \leq 1  \label{eq:diffZ}
\end{equation}%
for every set $A\subset \lbrack n]$ and $\ell =1,...,k.$ Indeed, consider
the first round, and assume that it ended at step $T_{1}=t$; i.e., $%
Z_{1}=X_{t}.$ Thus, the first $t-1$ steps were case (i), and step $t$ was
case (ii); i.e., $x_{0}^{1}+x_{0}^{2}<1$ in step $1$; then $%
(x_{0}^{1}+x_{0}^{2})+x_{0}^{3}$ in step $2$; and so on, up to $%
(x_{0}^{1}+...+x_{0}^{t-1})+x_{0}^{t}<1$ in step $t-1;$ and finally $%
(x_{0}^{1}+...+x_{0}^{t})+x_{0}^{t+1}\geq 1$ in step $t.$ Let $r\in
\{1,...,t\}$ be the \textquotedblleft winner" of the first $t-1$ steps,
i.e., $X_{t-1}^{r}=x_{0}^{1}+...+x_{0}^{t}=:\xi $ (for all other $i$ in $%
\{1,...t\}$ we have $X_{t-1}^{i}=0$). If the result of step $t$ is that $%
X_{t}^{r}=1$ and $X_{t}^{t+1}=\xi +x_{0}^{t}-1,$ then the only coordinate
that has increased from $Z_{0}=x_{0}$ to $Z_{1}=X_{t}$ is coordinate $r$,
and so the total weight that was moved is equal to the gain of $r,$ which is 
$1-x_{0}^{r}<1.$ If the result of step $t$ is that $X_{t}^{r}=\xi
+x_{0}^{t}-1$ and $X_{t}^{t+1}=1,$ then $t+1$ is not included in the set of
those that lost weight (because $t+1$ has surely gained weight), and so the
total weight that was moved, which equals the total weight that was lost, is
at most the sum of the weights of $1,...,t,$ which is $%
x_{0}^{1}+...+x_{0}^{t}=\xi <1.$ Thus in every situation the total weight
that was moved is $<1,$ and so $\left\vert Z_{1}^{A}-Z_{0}^{A}\right\vert <1$
for every set $A\subseteq \lbrack n].$ The same argument applies to every
round, which proves (\ref{eq:diffZ}).

We can therefore apply the Fan, Grama, and Liu inequality (see (\ref%
{eq:azuma}) in the Appendix) with $v=k\eta $, $c=k\delta ,$ and $T=k,$
yielding%
\begin{eqnarray*}
\mathbb{P}\left[ \frac{1}{k}\left\vert S\cap A\right\vert \geq w^{A}+\delta %
\right] &=&\mathbb{P}\left[ X_{T}^{A}-X_{0}^{A}\geq k\delta \right] \\
&\leq &\exp \left( -D\left( \frac{k\eta +k\delta }{k\eta +k}\left\Vert \frac{%
k\eta }{k\eta +k}\right. \right) \cdot k\right) ,
\end{eqnarray*}%
which equals $\pi ^{\ast }(\eta ,\delta ,k),$ as claimed. Again, the
martingale $-X_{t}^{A}$ provides the second inequality.
\end{proof}

\bigskip

\noindent \textbf{Remarks. }\emph{(a)} One can carry out the sequence of
steps of one round with a single randomization, as follows. Let the weights
involved be, in order, $x^{1},...,x^{t},x^{t+1};$ thus, $\xi
:=x^{1}+...+x^{t}<1$ and $\xi +x^{t+1}\geq 1$ (cf. the proof above). Then
for each $r$ in $\{1,...,t\},$ with probability $(x^{r}/\xi )\cdot (1-\xi
)/(2-\xi -x^{t+1})$ the weights at the end of the round are $\tilde{x}%
^{r}=1, $ $\tilde{x}^{t+1}=\xi +x^{t+1}-1$, and $\tilde{x}^{i}=0$ for all
other $i$ in $\{1,...,t\},$ and with probability $(x^{r}/\xi )\cdot
(1-x^{t+1})/(2-\xi -x^{t+1})$ they are $\tilde{x}^{r}=\xi +x^{t+1}-1,$ $%
\tilde{x}^{t+1}=1$, and $\tilde{x}^{i}=0.$ Thinking of this as
\textquotedblleft contests" where the \textquotedblleft winner" gets as much
weight as possible from the \textquotedblleft loser," the procedure first
selects the winner $r$ among $\{1,...,t\}$---with probabilites proportional
to the weights---and then the winner between $r$ (whose weight is now $\xi
=x^{1}+...+x^{t}$) and $t+1$---with probabilities proportional to $1-\xi $
and $1-x^{t+1}.$

Doing this for all rounds yields a procedure $\mathfrak{X}^{\ast \ast }$
that takes at most $k$ time periods.

\emph{(b) }The \textquotedblleft in order" requirement on $\mathfrak{X}%
^{\ast }$ is used to ensure that in every round, the active pair $\{i,j\}$
in each step consists of the \textquotedblleft winner" of the previous step
together with a new element. However, as each round is considered
separately, there is no need for this order to be kept between rounds (and
so the round does not have to start with the winner of the previous round).
That is, the requirement on $\mathfrak{X}^{\ast }$ is just to keep a
consistent sequential order \emph{within each round}.

\emph{(c)} The fact that the martingale $Z_{\ell }$ of $\mathfrak{X}^{\ast }$
has at most $k$ time periods in which the martingale differences are bounded
by $1$ (rather than $n$ time periods for the martingale $X_{t}$) yields an
immediate exponential bound, without having to estimate the sum of variances 
$V_{T}$: the Azuma--Hoeffding inequality ((\ref{eq:azuma}) in the Appendix,
with $T=k$) gives%
\begin{eqnarray*}
\mathbb{P}\left[ \frac{1}{k}\left\vert S\cap A\right\vert \geq \alpha
+\delta \right] &=&\mathbb{P}\left[ Z_{k}^{A}-Z_{0}^{A}\geq k\delta \right]
\\
&\leq &\exp \left( -\frac{(k\delta )^{2}/2}{k}\right) =\exp \left( -\frac{1}{%
2}\delta ^{2}k\right) .
\end{eqnarray*}%
Of course, Theorems \ref{th:freedman} and \ref{th:fgl} yield better bounds.

\section{Using the Bounds}

We provide two comments on using the above bounds. Note that both $\pi $ and 
$\pi ^{\ast }$ are increasing functions of their first variable\footnote{%
This follows, for instance, from the arguments in Section 3 of Fan, Grama,
and Liu [2012]: for $\pi ^{\ast }$ by equation (41) and Lemma 3.2, and for $%
\pi $ by equation (43) (with their $v^{2}/n$ as our $\eta $).} $\eta $.

\bigskip

\textbf{The best of }$A$\textbf{\ and its complement}\emph{. }Let $%
B:=[n]\backslash A$ be the complement of $A;$ its relative weight is $%
w^{B}=1-w^{A}=1-\alpha .$ Since $(1/k)\left\vert S\cap A\right\vert \geq
\alpha +\delta $ if and only if $(1/k)\left\vert S\cap B\right\vert \leq
1-\alpha -\delta ,$ we can take the better of the two bounds $\pi ^{\ast
}(\eta ^{A},\delta ,k)$ and $\pi ^{\ast }(\eta ^{B},\delta ,k)$ (where $\eta
^{A}:=\alpha -k\sum_{i\in A}(w^{i})^{2}\leq \alpha $ and $\eta
^{B}:=1-\alpha -k\sum_{i\in B}(w^{i})^{2}\leq 1-\alpha $), and get%
\begin{equation*}
\mathbb{P}\left[ \frac{1}{k}\left\vert S\cap A\right\vert \geq \alpha
+\delta \right] \leq \pi ^{\ast }(\min \{\eta ^{A},\eta ^{B}\},\delta ,k)
\end{equation*}%
(by the monotonicity of $\pi ^{\ast }$ in $\eta $); the same for $\mathbb{P}%
\left[ (1/k)\left\vert S\cap A\right\vert \leq \alpha -\delta \right] ,$ and
also for $\pi $. Since $\min \{\eta ^{A},\eta ^{B}\}\leq \min \{\alpha
,1-a\}\leq 1/2,$ we immediately get a uniform bound that is independent of%
\footnote{%
The inequality is by Pinsker's inequality $D(q||p)\geq 2(q-p)^{2}$. There
are sharper inequalities; for instance, it can be shown that $%
D(q||p)/(q-p)^{2}$ is minimized at $q=1-p,$ which yields 
\begin{equation*}
D(q||p)\geq C(p)~(q-p)^{2}
\end{equation*}%
for%
\begin{equation*}
C(p)=\frac{D(1-p||p)}{(1-2p)^{2}}=\frac{1}{1-2p}\ln \left( \frac{1-p}{p}%
\right) .
\end{equation*}%
In our case we get $D(1/3+2\delta /3||1/3)\geq \gamma \delta ^{2}$ for $%
\gamma =(2/3)^{2}C(1/3)=4\ln (2)/3\approx 0.92$, and so $8/9$ may be
increased to $\gamma .$} $A$:%
\begin{equation*}
\pi ^{\ast }\left( \frac{1}{2},\delta ,k\right) =\exp \left( -D\left( \frac{1%
}{3}+\frac{2\delta }{3}\left\Vert \frac{1}{3}\right. \right) \cdot k\right)
\leq \exp \left( -\frac{8}{9}\delta ^{2}k\right) ;
\end{equation*}%
cf. $\exp (-2\delta ^{2}k)$ in (\ref{eq:with-repl}).

\bigskip

\textbf{Bounds on}\emph{\ }$\eta \mathbf{.}$ One needs to estimate $\eta
=\alpha -k\sum_{i\in A}(w^{i})^{2}$ (see (\ref{eq:eta})); more precisely,
bound it from above (because, again, $\pi $ and $\pi ^{\ast }$ are
increasing in $\eta $). An immediate bound is of course $\eta \leq \alpha .$
For a better bound, assume that $A$ contains at most $m$ elements\footnote{%
A trivial bound for $m$ is, of course, $m\leq n$; a slightly better one is $%
m\leq n-\left\lceil k(1-\alpha )\right\rceil $, because the number of
elements of $N\backslash A$, whose relative weight is $w^{N\backslash
A}=1-\alpha ,$ is at least $(1-\alpha )/(1/k)$ (since each $w^{i}$ is at
most $1/k$).} (i.e., $\left\vert A\right\vert \leq m$); then $\sum_{i\in
A}(w^{i})^{2}\geq m(\alpha /m)^{2}=\alpha ^{2}/m$ (by the convexity of the
function $x\mapsto x^{2}$), and so\footnote{%
A precise computation yields $\eta =\alpha -(k/m)\alpha ^{2}-(km)\beta $
where $m=\left\vert A\right\vert ,$ and $\beta :=(1/m)\sum_{i\in
A}(w^{i}-\alpha /m)^{2}=(1/m)\sum_{i\in A}(w^{i})^{2}-\alpha ^{2}/m^{2},$
the variance of the weights of $A.$}%
\begin{equation*}
\eta \leq \alpha -\frac{k}{m}\alpha ^{2}=:\bar{\eta}_{\alpha ,m}
\end{equation*}%
(with $\bar{\eta}_{\alpha ,\infty }=\alpha $ corresponding to the trivial
bound $\eta \leq \alpha $).

\bigskip

For example, take $k=100,$ $\alpha =1/5,$ and $\delta =1/3-1/5;$ thus, we
want to estimate the probability that a set that has $1/5$ of the weight
gets more than a $1/3$ of the sample. The upper bounds on this probability,
using $\bar{\eta}_{\alpha ,m}$, are as follows:%
\begin{equation*}
\begin{tabular}{|c||c|c|c|c|}
\hline
& $m=\infty $ & $m=1000$ & $m=100$ & $m=50$ \\ \hline\hline
With replacement [(\ref{eq:with-repl})] & $0.0077$ & $0.0077$ & $0.0077$ & $%
0.0077$ \\ \hline
Procedure $\mathfrak{X}$ [$\pi $] & $0.0249$ & $0.0233$ & $0.0117$ & $0.0037$
\\ \hline
Procedure $\mathfrak{X}^{\ast }$ [$\pi ^{\ast }$] & $0.0212$ & $0.0198$ & $%
0.0097$ & $0.0029$ \\ \hline
\end{tabular}%
\end{equation*}

\appendix

\section{Appendix. Tail Probability Bounds\label{appendix:tail}}

Let $(X_{t})_{t\geq 0}$ be a martingale adapted to the sequence $(\mathcal{F}%
_{t})_{t\geq 0}$ of $\sigma $-fields, let $Y_{t}:=X_{t}-X_{t-1}$ be the
martingale differences, and $V_{t}:=\sum_{s=1}^{t}\mathrm{Var}\left[ Y_{s}|%
\mathcal{F}_{s-1}\right] =\sum_{s=1}^{t}\mathbb{E}\left[ Y_{s}^{2}|\mathcal{F%
}_{s-1}\right] $ the sum of their conditional variances. Recall that $%
D(q||p) $ denotes the Kullback--Leibler divergence from $p$ to $q.$ The
following inequalities hold for every $T\geq 1$ and $c,v>0.$

\begin{itemize}
\item \emph{Chernoff--Hoeffding} \emph{inequality}. Let $Y_{t}$ be i.i.d. 
\textrm{Bernoulli}$(p)-p$ for some $0<p<1$ (we subtract $p$ so that\footnote{%
We are stating the inequality in a way that parallels the martingale
inequalities that follow. What (\ref{eq:ch-ho}) says is that 
\begin{equation*}
\mathbb{P}\left[ \frac{1}{T}\mathrm{Binomial}(T,p)\geq p+\delta \right] \leq
\exp (-D(p+\delta ||p)T)\leq \exp (-2\delta ^{2}T)
\end{equation*}%
for every $T\geq 1,$ $0<p<1,$ and $\delta >0.$} $\mathbb{E}\left[ Y_{t}%
\right] =0$), and thus $X_{T}-X_{0}$ is \textrm{Binomial}$(T,p)-Tp$; then%
\begin{equation}
\mathbb{P}\left[ X_{T}-X_{0}\geq c\right] \leq \exp \left( -D\left( \left. p+%
\frac{c}{T}\right\Vert p\right) \cdot T\right) \leq \exp \left( -\frac{2c^{2}%
}{T}\right) .  \label{eq:ch-ho}
\end{equation}%
See Theorem 1 and Example 3 in Chernoff [1952], and Theorem 1 in Hoeffding
[1963].

\item \emph{Azuma--Hoeffding inequality}. Let $\left\vert Y_{t}\right\vert
\leq 1$ for every $t\geq 1$; then%
\begin{equation}
\mathbb{P}\left[ X_{T}-X_{0}\geq c\right] \leq \exp \left( -\frac{c^{2}/2}{T}%
\right) .  \label{eq:azuma}
\end{equation}%
See Theorem 2 in Hoeffding [1963] and Azuma [1967].

\item \emph{Freedman inequality}. Let $|Y_{t}|\leq 1$ for every $t\geq 1$;
then%
\begin{equation}
\mathbb{P}\left[ X_{T}-X_{0}\geq c\right] \leq \left( \frac{v}{v+c}\right)
^{v+c}\exp (c)\leq \exp \left( -\frac{c^{2}/2}{v+c/3}\right) ,
\label{eq:freedman}
\end{equation}%
for any bound $v$ on $V_{T}$ (i.e., $v\geq V_{T}$ a.s.). See Theorem (1.6)
in Freedman [1975], which gives\footnote{%
The second bound is a slight improvement over the one given in Freedman's
paper; see Theorem 1.1 in Tropp [2011] or Remark 2.1 in Fan, Gram, and Liu
[2012].} the above probability bound for the event $\{X_{t}-X_{0}\geq c$ and 
$V_{t}\leq v$ for some $t\},$ and so implies (\ref{eq:freedman}) when $%
V_{T}\leq v$ (a.s.).

\item \emph{Fan, Grama, and Liu} \emph{inequality.} Let $Y_{t}\leq 1$ for
every $t\geq 1$; then%
\begin{equation}
\mathbb{P}\left[ X_{T}-X_{0}\geq c\right] \leq \exp \left( -D\left( \frac{v+c%
}{v+T}\left\Vert \frac{v}{v+T}\right. \right) \cdot T\right) ,
\label{eq:fgl}
\end{equation}%
for any bound $v$ on $V_{T}$. See Theorem 2.1 in Fan, Grama, and Liu [2012],
which gives the above probability bound\footnote{%
Fan \emph{et al.} write their bound as%
\begin{equation*}
\left[ \left( \frac{v}{v+c}\right) ^{v+c}\left( \frac{T}{T-c}\right) ^{T-c}%
\right] ^{\frac{T}{v+T}},
\end{equation*}%
which is easily seen to be the same as the righthand side of (\ref{eq:fgl}).
We prefer the expression with the Kullback--Leibler divergence as it makes
it easier to compare with, say, (\ref{eq:ch-ho}).} for the larger event $%
\{X_{t}-X_{0}\geq c$ and $V_{t}\leq v$ for some $t\leq T\}$.
\end{itemize}

The final inequality (\ref{eq:fgl}) implies all the previous ones (see
Remark 2.1 and Corollary 2.1 in Fan, Grama, and Liu [2012]), and it reduces
to the Chernoff--Hoeffding inequality (\ref{eq:ch-ho}) in the i.i.d.
Bernoulli case.\footnote{%
For i.i.d. $Y_{t}\sim \mathrm{Bernoulli}(p)-p$ apply (\ref{eq:fgl}) to $\hat{%
Y}_{t}=Y_{t}/(1-p)\leq 1,$ with $\hat{c}=c/(1-p)$ and $\hat{v}=Tp/(1-p)$.}

\end{document}